\documentclass[12pt,reqno]{article}
\usepackage{subfiles}
\usepackage{amsmath,amsthm,amssymb,amsfonts,amsxtra,scalerel,graphicx}
\usepackage{pifont}

\usepackage{graphicx}
\usepackage{tikz-cd}

\usepackage{amscd}
\usepackage{multicol}
\usepackage{epstopdf}
\usepackage{color}
\usepackage{xparse}
\usepackage{comment}
\usepackage{mathrsfs}

\usepackage{epsf}
\usepackage{float}
\usepackage{extarrows}
\usepackage[a4paper, text={.8\paperwidth,.83\paperheight}, ratio=1:1]{geometry}
\usepackage[format=hang,font=small,textfont=it]{caption}
\usepackage{authblk}

\usepackage[english]{babel}
\usepackage{latexsym}
\usepackage{nicefrac}
\usepackage{graphicx}

\usepackage{lipsum}

\usepackage{authblk}

\usepackage{tikz}
\usetikzlibrary{calc}
\usepackage[tight]{subfigure}
\usepackage{afterpage}
\usepackage{hyperref}

\newcommand{\CP}{\mathbb{CP}^2}
\newcommand{\C}{\mathcal{C}}
\newcommand{\e}{\varepsilon}
\newcommand{\z}{\zeta}

\usepackage{color}
\definecolor{greenbean}{RGB}{199,237,204}
\newtheorem{thm}{Theorem}[section]
\newtheorem{lemma}[thm]{Lemma}
\newtheorem{prop}[thm]{Proposition}
\newtheorem{cor}[thm]{Corollary}

\theoremstyle{definition}
\newtheorem{Def}[thm]{Definition}
\newtheorem{eg}{Example}[section]
\newtheorem{theorem}{Theorem}[section]
\newtheorem{question}[theorem]{Question}
\theoremstyle{definition}

\newtheorem{rmk}[thm]{Remark}

\begin{document}
\title{Fundamental groups of highly symmetrical curves and Fermat line arrangments}
%\author{Meirav Amram, Praveen Kumar Roy and Uriel Sinichkin}

\footnotetext{Email address: Meirav Amram: meiravt@sce.ac.il, Praveen Roy: praveen.roy1991@gmail.com, Uriel Sinichkin: sinichkin@mail.tau.ac.il }

%\author{Meirav Amram$^1$, Praveen Kumar Roy$^{1,2}$, Uriel Sinichkin$^3$}

\author[1]{Meirav Amram}
\author[2]{Praveen Kumar Roy}
\author[3]{Uriel Sinichkin}

\affil[1]{\small{Department of Mathematics, Shamoon College of Engineering, Ashdod, Israel}}
\affil[2]{\small{UM-DAE Centre for excellence in basic sciences, Mumbai, India}}
\affil[3]{\small{School of Mathematical Sciences, Tel Aviv University, Tel Aviv, Israel}}
	
	\date{}%\today
	
	\maketitle

\maketitle
\abstract{ We showcase a computation of the fundamental group of $\mathbb{CP}^2 - \mathcal{C}$ when $\mathcal{C}$ is a curve admitting a lot of symmetries.
In particular, let $\mathcal{C}$ denote the Fermat line arrangement in $\mathbb{CP}^2$ defined by the vanishing 
locus of homogeneous polynomial $(x^n-y^n)(y^n-z^n)(z^n-x^n)$. In this article, we compute the fundamental group $\pi_1(\mathbb{CP}^2-\mathcal{C})$ of complement of this line arrangement in the complex projective plane. We show that this group is semi-direct product of $G$ and $F_n$, i.e., $\pi_1(\mathbb{CP}^2-\C, \overline{\epsilon}) = G \rtimes F_{n}$, where $G$ and $F_n$ is defined in \ref{defn of G}, and \ref{final theorem} respectively. 
%\mytodo{Add more to the abstract.}
}

\tableofcontents

\section{Introduction}

The complement of line arrangements in the projective plane, or more generally, the complement of hyperplane arrangements in an n-dimensional projective space, has been a focal point of interest for researchers working in the areas of topology and geometry since its inception. As remarked by Hirzebruch in \cite{Hir}, ``The topology of the complement of an arrangement of lines in the projective plane is very interesting, and the investigation of the fundamental group of the complement is very difficult." Indeed, computing these fundamental groups is often a challenging task.

Beginning with a relatively simple example, such as the complement of a collection of lines in the real plane, understanding its topology can be achieved by merely counting its connected components. However, when considering a set of lines in the complex plane, the complement becomes connected, making its fundamental group more intriguing. The complexity of these fundamental groups increases significantly, as seen in the case of $\mathbb{C}^l - \cup_{i=1}^{n}H_i$ (where $\{H_1, H_2,…, H_n \}$ are hyperplane arrangement). Nevertheless, there are methods to tackle these challenges. 

Given an arrangement $\mathcal{A} = \{ H_1, H_2,…, H_n \} \subset \mathbb{C}^3$ there corresponds an arrangement in $\overline{\mathcal{A}}  \subset \mathbb{CP}^2$. In \cite{Hir} Hirzebruch defined a compact, smooth, complex algebraic surface $M:= M_n(\mathcal{A})$ as the minimal desingularization of certain branched cover of $\mathbb{CP}^2$ (see \cite{Hir}). The author then proceeds to compute the Chern numbers of $M$.

\begin{comment}
\textcolor{blue}{I added here some info about fundamental groups and the significance of the curve, maybe you can add more?!} \textcolor{orange}{Praveen thinks that maybe not needed}
Fundamental group are interesting objects, as they are considered to be invariants of classification of algebraic curves. 
They have been used over the years to find Zariski pairs and to classify curves according their degree, along with the constrains of existence of certain components, such as lines and conics, and their positions in the curves. Some singularities in Fermat arrangements are with high multiplicities, and there is a challenge to understand what and how they contribute to the related fundamental groups. It is interesting to determine the fundamental groups of the complement of Fermat arrangements, and in particular for the general case. Despite the importance of the above questions, there have been no recent significant developments in this area. The reasons may be methodological limitations or the inability to combine algebraic calculations together with geometrical progress. Classification of such family of curves is not easy, especially because of the concerns about high multiplicity of singularities and lack of results.
\end{comment}
Fermat arrangements, also known as Ceva arrangements, were first explored by F. Hirzebruch in his pursuit of finding a minimal algebraic surface of general type with Chern numbers satisfying $c_1^2/c_2 \leq 3$ (see \cite{Hir}). The name ``Fermat line arrangement" comes probably from the fact that this arrangement is realized as singular locus of the pencil generated by $(x^n-y^n)$ and $(y^n-z^n)$ whose non-singular members are isomorphic to Fermat curve $x^n + y^n + z^n =0$ \cite{Urzua}. 
In the theory of algebraic surfaces, this line arrangement garnered significant attention due to its relation to the containment problem. More precisely, the ideal of points (in $\mathbb{CP}^2$), dual to the Fermat line arrangement, gives a counter example to the following question
\begin{question}(Huneke)
Let $I$ be an ideal of points in $\mathbb{P}^2$. Is there a containment 
\[
I^{(3)} \subseteq I^2 ?
\]
\end{question}
Here $I^{(3)}$ denotes the $3$rd symbolic power of the ideal $I \subset R:= \mathbb{C}[x,y,z]$, defined by 
\[
I^{(3)} := R \cap \left(\bigcap_{p \in Ass(I)} I^{3} R_p\right).
\]
This counter-example was found by Dumnicki, Szemberg and Tutaj-Gasi\'nska in \cite{DST}, and it is the first counter-example appearing in the literature. This arrangement is further generalized to Fermat-type arrangements of hyperplanes in higher-dimensional projective spaces, and it serves as an example of non-containment problems in general. 

On another note, Fermat-type arrangements are connected with what is known as the ``Bounded Negativity Conjecture (BNC)". BNC predicts that the self intersection of reduced and irreducible curves on a surface is bounded from below \cite{BNC-1}. Later, in \cite{BNC-2}, the authors established a link between the bounded negativity conjecture and Fermat-type arrangements. They also demonstrate that Fermat-type arrangements provide examples of extremal Harbourne constants. 

Furthermore, Fermat line arrangements hold special significance in the theory of multi-points (corresponding to singular points of line arrangements) Seshadri constants \cite{Pokora}.

We return now to the main motivation of this paper, which is the study of fundamental groups of complements of line arrangements. 
Before we focus on our methods and explanations, we mention that there are results on monodromy and fundamental groups. Some of these results can be found in \cite{Quadric-L-A,cog1,Garber-Teicher,OS,Randell,Sa}. 
All of those results are based on the foundational work of Van Kampen \cite{vk1}.
Those techniques are using the well developed theory of braids (\cite{Cohen-Sucio,Orlik-Ter}), which, in turn, is connected to the theory of knots and links in 3-space (\cite{Falk-Rand}). We mention works \cite{Art4} and \cite{cog1} in addition to \cite{Cohen-Sucio} as foundational references for understanding the study of monodromy and fundamental groups, highlighting the strong correspondence between them.

The fundamental group of the complement of a line arrangement has primarily been studied using the theory developed by Moishezon and Teicher in a series of papers \cite{Moish-Teicher-I,Moish-Teicher-II}.  Unfortunately, those results are not well suited for the study of the fundamental groups of Fermat's arrangements, for two reasons.
The main reason lies in the fact that the computation uses a generic projection $ \mathbb{C}^2 \to \mathbb{C}$ which is injective on the set of singular points of the arrangement.
If one wishes to perform such a computation, the number of singular points will increase significantly as the degree of the arrangement increases, with no simple description of their relative position after the projection.
After a preliminary investigation, it appears that the computation becomes very hard even in low degrees, and the resulting representation of the fundamental group lacks useful symmetries which might help its simplification.
Another problem in applying the results of Moishezon and Teicher lies in the fact that they perform the computation for a real algebraic curve, all of which singularities are real, which is not true for Fermat's arrangements with $n\ge 3$.
This issue could be managed, for example by using Van Kampen's original work, but it adds an additional layer of difficulty.

There are independent studies on this topic, such as the work presented in \cite{Fan1}, where the author proves
\begin{prop}[Fan]
    Let $\Sigma = \cup l_i$ be a line arrangement in $\mathbb{CP}^2$ and assume that there is a line $L$ of $\Sigma$ such that for any singular point $S$ of $\Sigma$ with multiplicity $\geq 3$, we have $S \in L$. Then $\pi_1(\mathbb{CP}^2 -\Sigma)$ is isomorphic to a direct product of free groups.
\end{prop}
In \cite{Fan2}, the author conducted similar studies and established the following:
\begin{prop}[Fan]
    Let $\Sigma$ be an arrangement of $n$ lines and $S = \{a_1,...,a_k\}$ be the set of all singularities of $\Sigma$ with multiplicity $\geq 3$. Suppose that $\beta(\Sigma) = 0$, where $\beta(\Sigma)$ is the first Betti number of the subgraph of $\Sigma$ which contains only the higher singularities (i.e. with multiplicity $\geq 3$) and their edges. Then:
    \[
    \pi_1(\mathbb{CP}^2 -\Sigma) = \mathbb{Z}^r \oplus \mathbb{F}^{m(a_1) - 1} \oplus \cdots \oplus \mathbb{F}^{m(a_k) - 1} 
    \]
    where $r=n+k-1-m(a_1)- ... - m(a_k)$.
\end{prop}
Those results are also not applicable for Fermat's arrangement, since there are $3$ singular points of multiplicity $\ge 3$, all of them connected to each other by a line.

It is thus desirable to develop a method for the computation of the fundamental groups of line arrangements, which is better suited to account for the various symmetries and nice structure Fermat's arrangements possess, which is the main result of the current work.
The main idea is to replace the generic projection which is injective on the set of singular points of the arrangement, by a carefully chosen projection which is not only non-injective on the singular points, but even parallel to some of the lines in the arrangement.
A key step in our technique involves identifying a subspace of the complement of the arrangement and defining a projection from that space that forms a fibration (see \ref{crucial_lemma}). This process is repeated multiple times in the paper. Additionally, we make repeated use of the Seifert-Van Kampen Theorem \cite{vk2} to compute the fundamental group.

To begin, we modify the line arrangement so that the projection we define from the complement of this modified arrangement to $\mathbb{C}-\{n+1 \; \text{points}\}$ becomes a fibration. Next, we compute the fundamental group of the modified space by analyzing the monodromy action induced by the fundamental group of the base space on the fiber space of the projection. Some of these actions are quite intricate for general $n$ and require careful analysis.

Subsequently, we define two open sets of $\mathbb{CP}^2$ in such a way that the space of interest becomes the union of the modified space and these two defined open sets. The remainder of our computation is divided into two sections, each dealing with the process of attaching the open sets to the modified space consecutively and computing the fundamental group of the resulting space using a new projection that we define, along with the Seifert-Van Kampen Theorem.

Finally, we employ the Seifert-Van Kampen Theorem once again to compute the fundamental group that is of interest to us.

\begin{comment}
Our method to compute the fundamental group in question, is quite different compared to the traditional method of computation for 
such problems, which are mainly using the theory of braid techniques developed by Moischoin-Tiecher. The crucial steps in our technique are to first identify a space and define a projection from that space so that it becomes fibration of topological space \ref{crucial_lemma}, and we do this process multiple times in the paper. Along with it we also use Van-Kampan theorem repeatedly to compute the fundamental group. We first modify the line arrangement so that the projection we define from the complement of this modified arrangement to 
$\mathbb{C}-\{n+1 \; \text{points}\}$, is a fibration. Afterwards we compute the fundamental group of the modified space by computing the  monodromy action induced by the fundamental group of the base space on the fiber space of the projection. Some of these actions are quite complicated for general $n$ and require careful analysis. We then define two open sets of $\mathbb{CP}^2$ such that the space of our 
interest is equal to the union of the modified space and these two defined open sets. We divide the rest of our computation into two section wherein each section deals with the details of attaching the open sets to the modified space (consecutively), and computing the fundamental group of resulting space using Van-Kampan theorem and a new projection which we define. Finally we use the Van-Kampen theorem once again to compute the fundamental group that we are interested in. 
\end{comment}

The result of the computation of the fundamental group is as follows, where $G$ is defined in \ref{defn of G}, and $F_n$ is the free group on $n$ generators.
\begin{theorem}\label{final theorem}
    The fundamental group of the complement of Fermat line arrangement is semi-direct product of $G$ and $F_n$ i.e.,
    \[
    \pi_1(\mathbb{CP}^2-\C, \overline{\epsilon}) = G \rtimes F_{n}.
    \]
\end{theorem}

It is worth noting here that although fundamental groups of complements of curves are used frequently in constructing Zariski pairs \cite{Art1,Art2,Art3}, Fermat arrangement is not a part of a Zariski pair.
Indeed, If any line arrangement has the same combinatorics, we can use a projective transformation to put the multiple nodes at the intersections of the coordinate axes, and then use a continuous deformation to put all the double nodes at the positions they occupy in the Fermat's arrangement.

The organization of paper is as follows: In Section \ref{method}, we provide a detailed explanation of the computation method. More specifically, in Subsection \ref{Notations and setting} we introduce the necessary definitions and notations that will be used throughout the paper and in Subsection \ref{computational-method}, we write down the whole computational method concisely. Additionally, in Subsection \ref{Notations and setting}, we define two open sets, $U_0$ and $U_\infty$, and proceed to compute their fundamental groups in Subsections \ref{fundamental group of U0} and \ref{Fundamental group of Uinfinity}, respectively. Finally, we compute the fundamental group in Section \ref{final computation}.

\section*{Acknowledgements}
The second named author was supported by the postdoctoral fellowship of SCE during a part of this work.

\section{Definitions and description of the computational method}\label{method}

\subsection{Notations and setting}\label{Notations and setting}

We will work in the projective plane with coordinates $x,y,z$, and denote by $V(f)$ the vanishing locus in $\mathbb{CP}^2$ of a homogeneous polynomial $ f\in \mathbb{C}[x,y,z] $.
The $n$'th Fermat's arrangement, which we denote by $\mathcal{C}$, is defined to be  
\[
V\left( (x^n-y^n)(y^n-z^n)(z^n-x^n) \right) \subset \mathbb{CP}^2.
\]
We will compute the fundamental group of $\mathbb{CP}^2 - \mathcal{C}$ with respect to the basepoint $\overline{\epsilon} = [1:\varepsilon:0] \in \mathbb{CP}^2$.

The curve $\mathcal{C}$ is a union of $3n$ lines: 
\[ \mathcal{C} = \bigcup_{k=0}^{n-1} L_{x,k}  \cup  \bigcup_{k=0}^{n-1} L_{y,k} \cup \bigcup_{k=0}^{n-1}L_{z,k}, \]
where we denote $L_{x,k} := V(y - \zeta_n^k z)$, $L_{y,k} := V(z - \zeta_n^k x)$, and $L_{z,k} := V(x - \zeta_n^k y)$, and $\zeta_n=\exp\left(\frac{2\pi i}{n} \right) $ is a primitive root of unity.
Also, let 
\begin{align*}
	q_z := \bigcap_{k=0}^{n-1}L_{z,k} = [0:0:1] ; \quad q_x := \bigcap_{k=0}^{n-1}L_{x,k} = [1:0:0]; \cr
	\text{and} \quad  q_y := \bigcap_{k=0}^{n-1}L_{y,k} = [0:1:0].
\end{align*}

The main idea of the computation will be to define the projection 
\begin{align}\label{original_proj}
Pr : \mathbb{CP}^2-(\mathcal{C}\cup V(x) \cup V(y)) & \longrightarrow \mathbb{C}-\{0,1,\zeta_n, \zeta_n^2,...,\zeta_n^{n-1}\}\\
Pr ([x:y:z]) & \longmapsto y/x. \nonumber
\end{align}
It would be easy to compute the fundamental group of $\mathbb{CP}^2 -(\mathcal{C}\cup V(x)\cup V(y))$ by using the fact $Pr$ is a fibration (see Corollary \ref{cor:fundamental_group_of_U0_intersect_U_infty}).

To compute the fundamental group of the full space of interest, we introduce two open sets $U_0$ and $U_{\infty}$ containing $V(y)$ and $V(x)$, respectively, as follows: 
\begin{eqnarray*}
	U_0 &:=& \{[x:y:z] \in \CP-\C \ | \ x\ne 0, \; \left|\frac{y}{x}\right| < 2\e \}, \ \text{and} \\
	U_\infty &:=& \{[x:y:z] \in \CP-\C \ | \ y\ne 0, \;  \left|\frac{x}{y}\right| < 2\e \}.
\end{eqnarray*}
Note that if we attach $U_0$ and $U_\infty$ to $\mathbb{CP}^2-(\mathcal{C}\cup V(x) \cup V(y)$ we get the desired space, i.e., 
\[
\mathbb{CP}^2-\mathcal{C} = (\mathbb{CP}^2-(\mathcal{C}\cup V(x) \cup V(y))) \cup U_0 \cup U_{\infty}.
\]

Finally, we denote by $``."$ to be the group action and by $``\cdot"$ to be the product.

\subsection{The computational method}\label{computational-method}

We begin by refining our line arrangement by adjoining $V(x)$ and $V(y)$ into it. 
This enables us to construct the fibration $Pr$ defined in \eqref{original_proj}.
The fact that $Pr$ is a fibration allows us to easily compute the fundamental group of $ \mathbb{CP}^2 - (\mathcal{C} \cup V(x) \cup V(y))$ by using the following lemma. 
\begin{lemma}\label{crucial_lemma}
Let $p : X \to B$ be a fibration of path-connected topological spaces $X$ and $B$ and let $b \in B$ be a fixed point. Let $F$ denotes the fiber $p^{-1}(b)$ with the inclusion map $in:F\to X$. Suppose that $F$ is path-connected and that $\pi_2(B)$ is trivial. Then the following is a short exact sequence of groups
\begin{equation}\label{eq:short_exact_sequance_pi_1}
1\longrightarrow \pi_1(F,x) \xlongrightarrow{in_*} \pi_1(X,x) \xlongrightarrow{p_*} \pi_1(B,b) \longrightarrow 1,
\end{equation}
for $x\in F$. Moreover, if $p$ has a section then 
\[
\pi_1(X,x) = \pi_1(B,b) \ltimes \pi_1(F,x).
\]
\end{lemma} 

\begin{proof}
	We have the well known long exact sequence in homotopy groups (see for example \cite[Theorem 4.41]{AH})
	\[ \ldots \longrightarrow \pi_n(F, x) \xlongrightarrow{in_*} \pi_n(X, x) \xrightarrow{p_*} \pi_n(B, b) \longrightarrow \pi_{n-1}(F, x) \longrightarrow \ldots \longrightarrow \pi_0(B, b) \longrightarrow 1. \]
	By our assumptions $ \pi_2(B)=1$ and $\pi_0(F) = 1$, so we get the short exact sequence \eqref{eq:short_exact_sequance_pi_1}.
	If $p$ has a section, it induces a section of $p_*$, so \eqref{eq:short_exact_sequance_pi_1} splits and thus $ \pi_1(X, x) = \pi_1(B, b) \ltimes \pi_1(F, x)$.
\end{proof}

\begin{cor}\label{cor:fundamental_group_of_U0_intersect_U_infty}
	The fundamental group of $\mathbb{CP}^2 - (\mathcal{C}\cup V(x) \cup V(y))$ is
	\begin{equation}\label{Fundamental group of $U$}
		\pi_1(\mathbb{CP}^2-(\mathcal{C}\cup V(x) \cup V(y)), \overline{\epsilon}) = \mathcal{F}_{n+1} \ltimes \mathcal{F}_{2n}.
	\end{equation}
\end{cor}

\begin{proof}
	The projection $Pr$ is a fibration. Its base is $\mathbb{C} - \{0, 1, \zeta_n, \zeta_n^2, \ldots, \zeta_n^{n-1}\}$ which is path connected and have a trivial second homotopy group (since it has a deformation retract to a graph).
	The fiber of $Pr$ is $\mathbb{C} - \{ 2n\text{ points}\}$ which is path connected.
	Thus the result follows from Lemma \ref{crucial_lemma}, since the map $ t \mapsto [1:t:0]$ is a section for $Pr$.
\end{proof}

\begin{Def}\label{Definition of $R_1$}
	We define $R_1$ to be the set of all relations that holds among the generators of $\pi_1(\mathbb{CP}^2-(\mathcal{C}\cup V(x) \cup V(y)), \overline{\epsilon})$. 
\end{Def}

Next we compute the action of the base fundamental group $\mathcal{F}_{n+1}$ on the fundamental group $\mathcal{F}_{2n}$ of the fiber space and list down all the relations. Detailed computations can be found in Lemma \ref{Prop for R1 relations} and Section \ref{Fundamental group of Uinfinity}. The main remaining task is to reattach $V(x)$ and $V(y)$ back into the space $\mathbb{CP}^2-\mathcal{C} \cup V(x) \cup V(y)$ and compute the required fundamental group. We accomplish this in two steps: first, we add $V(x)$, considering an open neighbourhood  $U_0$ of $V(x)$, compute its fundamental group, and then apply the Van Kampen Theorem on the open sets $U_0$ and 
$\mathbb{CP}^2-\mathcal{C} \cup V(x) \cup V(y)$ to obtain the fundamental group of $\mathbb{CP}^2-\mathcal{C} \cup V(y)$, based at the basepoint $\overline{\epsilon}$. 
Analogously, in the second step, we consider a neighbourhood $U_{\infty}$ of $V(y)$, compute its fundamental group, and use the open sets $U_{\infty}$ and $\mathbb{CP}^2-\mathcal{C} \cup V(x) \cup V(y)$ to apply the Van Kampen Theorem
and  compute the fundamental group of $\mathbb{CP}^2-\mathcal{C} \cup V(y)$. 
This step has the additional complexity that $U_{\infty}$ does not contain our basepoint, so we need to change the basepoint be conjugation with a path from $\overline{\epsilon}$ to a point inside $U_{\infty}$ (see Section \ref{final computation}).

Afterwards, we use the Van Kampen Theorem once again with open covers $\mathbb{CP}^2-(\mathcal{C} \cup V(x))$ and $\mathbb{CP}^2-(\mathcal{C} \cup V(y))$ of $\mathbb{CP}^2-\mathcal{C}$ to obtain the 
final result.

\section{Fermat arrangements}

\subsection{Fundamental group of $U_0$}\label{fundamental group of U0}
We first pick $U_0$ and compute $\pi_1(U_0, \overline{\epsilon})$. We write $U_0$ as a union of two open sets $U_{0,1}$ and $U_{0,2}$ defined as follows:
\begin{eqnarray*}
U_{0,1} &:=& \{ [x:y:z] \in U_0 \ | \ \left|\frac{z}{x}\right| < {\frac{1}{2}} \}, \ \text{and} \\
U_{0,2} &:=& \{ [x:y:z] \in U_0 \ | \ \left|\frac{z}{x}\right| > {\frac{1}{3}} \}.
\end{eqnarray*}

\subsubsection{The set $U_{0,1}$}
The set $U_{0,1}$ is homotopy equivalent to a neighbourhood $U_{0,1}'$ such that the map
\[
Pr_{0,1} : U_{0,1}' \longrightarrow B_{2\e}(0)-\{0\}
\]
defined by,
\[
 [x:y:z] \mapsto \frac{y-z}{x}
\]
is a fibration with fiber isomorphic to $\mathbb{C} - \{n-1 \ \text{points}\}$ (corresponding to the lines $L_{x,i}$, for $i=1,2,...,n-1$). 
We then get from Lemma \ref{crucial_lemma} the following short exact sequence of groups:
\[
1\longrightarrow \pi_1(\mathbb{C} - \{n-1 \ \text{points}\}, \overline{\epsilon}) \longrightarrow \pi_1(U_{0,1}, \overline{\epsilon}) \longrightarrow \pi_1(B_{2\e}(0)-\{0\}, \e) \longrightarrow 1.
\]
Let $g_i$ and $g_i'$  ($0\leq i \leq n-1$) be loops based at the basepoint $\overline{\epsilon} \in \CP$ going around the points $L_{x,i} \cap Pr_{0,1}^{-1}(\e)$ and $L_{y,i} \cap Pr_{0,1}^{-1}(\e)$ respectively. It is clear that the fundamental group $\pi_1(U_{0,1}, \overline{\epsilon})$ is generated by the loops $g_i$ with some relations. We need to find out these relations. In this step, first we construct loops $\overline{g_i}$ based at $\e$ and homotopic to $g_i$. To do this we consider a small number $\delta$ and  write down the equation of $g_i$ as follows:
\begin{eqnarray}\label{eqn of gi}
g_i(t) =
         \begin{cases}
            [1: \e: t\z_n^{n-i}(\e - \delta)], \quad 0 \leq t \leq 1 \\
            [1: \e: \z_n^{n-i}(\e - \delta \exp(2\pi i(t -1))], \quad 1 \leq t \leq 2 \\
            [1: \e: \z_n^{n-i}(\e - \delta)(3-t)], \quad 2 \leq t \leq 3.
         \end{cases}
\end{eqnarray}
We project these $g_i$ ($i \neq 0$) through $Pr_{0,1}$ and notice that these are null-homotopic in $B_{2\e}(0)-\{0\}$. So that, we can take this homotopy and lift it (by $Pr_{0,1}$) to $U_{0,1}$ to get loops contained in the fiber of $Pr_{0,1}$ and still homotopic to $g_i$. We denote these loops by $\overline{g_i}$ and their equations are 
\begin{eqnarray*}
\overline{g_i}(t) = 
                        \begin{cases}
                           [1 - (1 - \delta/\e)t\z_n^{n-i}: \e: (\e - \delta)t\z_n^{n-i}], \quad 0 \leq t \leq 1 \\
                           [1 - (1 - \delta/\e)\z_n^{n-i}\exp(2\pi i(t -1)): \e: \z_n^{n-i}(\e - \delta \exp(2\pi i(t -1)))], \quad 1 \leq t \leq 2 \\
                           [1 - (1 - \delta/\e)(3-t)\z_n^{n-i}: \e: \z_n^{n-i}(\e - \delta)(3-t)], \quad 2 \leq t \leq 3.                                                                    
                       \end{cases}
\end{eqnarray*}
This is a loop around the point $[(\z_n^i - 1) : \e\z_n^i : \e] = Pr_{0,1}^{-1}(\e) \cap L_{x,i}$. We fix the coordinate on $Pr_{0,1}^{-1}(\e)$ to be $y/x$, then these points become $\frac{\e\z_n^i}{\z_n^i - 1}$. It is not difficult to see that these points 
lie on the line $Re(y/x) = \e/2$. 
Note that $Pr_{0,1}(g_0)$ is a loop in $B_{2\e}(0)-\{0\}$ starting at $\e$ and going till $\delta$ in straight line and then making a counter clockwise loop around origin before returning to $\e$. We denote this loop too by $g_0$ and compute the monodromy action of $g_0$ on $g_i$ ($1\leq i \leq n-1$) as shown in Figure \ref{fig:action of $g_0$ in general case}.

\begin{figure}[H]
\begin{center}
\raggedleft
\begin{minipage}{13cm}
  \includegraphics[width=0.6\linewidth]{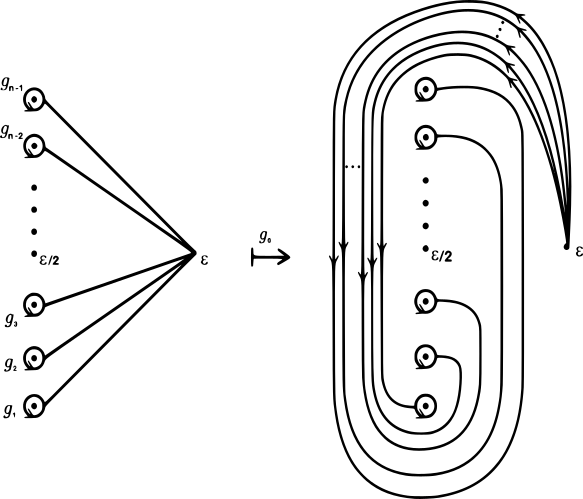}
  \end{minipage}
    \caption{Action of $g_0$ on $g_i$}
      \label{fig:action of $g_0$ in general case}
\end{center}
\end{figure}

Therefore, we get the following relations

\begin{eqnarray}\label{relation of $g_0$ on general n}
g_0 . g_i = g_{n-1}g_{n-2}...g_1g_ig_1^{-1}...g_{n-2}^{-1}g_{n-1}^{-1}
\end{eqnarray}
for $1\leq i \leq n-1$.
This shows that the fundamental group is
 \[
 \pi_1(U_{0,1}, \overline{\epsilon}) = \  \langle g_0,g_1,...,g_{n-1}\ | \ (\ref{relation of $g_0$ on general n})  \ \text{holds for $1\leq i \leq n-1$} \rangle.
 \]
 %\end{enumerate}
 
 \subsubsection{The set $U_{0,2}$} 
 Now we turn to $U_{0,2}$ and note that $U_{0,2}$ does not contain the basepoint $\overline{\epsilon}$. To include the basepoint we redefine $U_{0,2}$ by attaching a contractible strip ``$S$" between $g_0$ and $g_{n-1}$ as depicted in Figure \ref{Inclusion of a strip}, containing $\overline{\epsilon}$. Note that it does not change the 
 fundamental group of $U_{0,2}$.
 From now on, $U_{0,2}$ will contain $S$.  

\begin{figure}[H]
\begin{center}
%\raggedleft
\begin{minipage}{7cm}
  \includegraphics[width=0.7\linewidth]{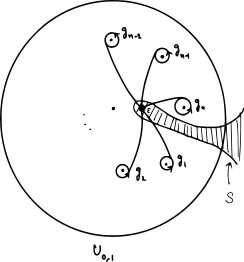}
  \end{minipage}
    \caption{Inclusion of a strip in $U_{0,2}$}
      \label{Inclusion of a strip}
\end{center}
\end{figure}

We now define the loop $g_i'$ precisely (note that it is contained in $U_{0,2}$ in view of Figure \ref{Inclusion of a strip}). Before that we define for all $0 \leq j \leq n-1$,
\begin{eqnarray}\label{eqn of gi-prime}
h_{j}'(t) = 
                        \begin{cases}
                           [1 : \e: \frac{t(1+\z_n^{n-1})}{2}], \quad 0 \leq t \leq 1 \\
                           [1 : \e: \frac{1}{2}(\exp(2\pi i (t-1)j/n)+(2-t)\z_n^{n-1})], \quad 1 \leq t \leq 2 \\
                           [1 : \e: (\frac{(3-t)}{2} + (t-2)(1-\delta))\z_n^{j}], \quad 2 \leq t \leq 3 \\
                           [1 : \e: (1-\delta\exp{2\pi i(t-3)})\z_n^j], \quad 3 \leq t \leq 4 \\                                                       
                           [1 : \e: (\frac{(t-4)}{2} + (t-3)(1-\delta))\z_n^{j}], \quad 4 \leq t \leq 5 \\                                                                        
                           [1 : \e: \frac{1}{2}(\exp(2\pi i (6-t)j/n)+(t-5)\z_n^{n-1})], \quad 5 \leq t \leq 6 \\
                           [1 : \e: \frac{(7-t)(1+\z_n^{n-1})}{2}], \quad 6 \leq t \leq 7.                    \end{cases}
\end{eqnarray}
For all our computations in this paper, we use $g_i'$, as shown in all the figures, wherever it appears from now on. Therefore 
it is clear from those figures and from the equation defined for $h_i'$ that 
\begin{eqnarray*}
    g_0' &=& g_0^{-1}h_0'g_0 \\
    g_1' &=& g_{n-1}^{-1}g_0^{-1}h_1'g_0g_{n-1} \\ 
    &\vdots& \\
    g_{n-1}' &=& g_1^{-1} \cdots g_{n-1}^{-1}g_0^{-1}h_{n-1}'g_0g_{n-1} \cdots g_1.
\end{eqnarray*}
Although we do not use the above relations anywhere in our computation, we have given it for the sake of complete 
information for the readers. We now note the following

\begin{lemma}\label{geometry of U0,2}
$U_{0,2}$ is homotopy equivalent to $\mathbb{C}-\{n+1 \ \text{points} \}$.
\end{lemma}
\begin{proof}
	Since $ x\ne 0 $ in $ U_{0,2} $ we can switch to affine coordinates and see that 
	\begin{align*}
		U_{0,2} & \cong \left\{(y', z')\in \mathbb{C}^2 \; | \; |y'|<2\varepsilon, |z'|>\frac{1}{3}, (y')^n \ne 1, (z')^n \ne 1, (y')^n \ne (z')^n \right\} = \\
		& = \left\{(y', z')\in \mathbb{C}^2 \; | \; |y'|< 2\varepsilon, |z'|>\frac{1}{3}, (z')^n \ne 1 \right\} = \\
		& = B_{2\varepsilon}(0) \times \left( \mathbb{C} - \left( \overline{B_{1/3}(0)}\cup \left\{\zeta_n^k \; | \; k=0,\dots,n-1 \right\} \right) \right).
	\end{align*}
	Since $ B_{2\varepsilon}(0) $ is contractible, we get the desired result.
\end{proof}

From the definition of $U_{0,2}$ and the above lemma, it is thus clear that the fundamental group $\pi_1(U_{0,2}, \overline{\epsilon})$ is the free group generated by $g_j'$ for $j = 0,1,2,...,n-1$ and one additional generator, we now express this new generator as a product of $g_0, \dots, g_{n-1}$.
In fact it is clear from the Figure \ref{Inclusion of a strip}, that the additional generator is equal to $g_0g_{n-1}...g_1$, i.e., a product of $g_j$ in cyclic order starting from $g_0$. We now show that all such cyclic products starting from any $g_j$, are one and same. To this end we note from (\ref{relation of $g_0$ on general n}), the following
\begin{eqnarray}\label{gi commutes for all i}
    g_0^{-1}g_jg_0 &=& g_{n-1}g_{n-2}...g_1g_jg_1^{-1}...g_{n-2}^{-1}g_{n-1}^{-1} \nonumber \\ 
    &\Rightarrow& [g_i,\ g_0g_{n-1}g_{n-2}...g_1] = e,
\end{eqnarray}
for all $1\leq j \leq n-1$. This, in particularly for $j=1$, gives
\begin{eqnarray}\nonumber\label{cyclic:1}
    g_1g_0g_{n-1}g_{n-2}...g_1 = g_0g_{n-1}g_{n-2}...g_1g_1 \\
    \Rightarrow g_1g_0g_{n-1}g_{n-2}...g_2 = g_0g_{n-1}g_{n-2}...g_1.
\end{eqnarray}
Substituting this value in (\ref{gi commutes for all i}), we see that
\[
[g_j,\ g_1g_0g_{n-1}g_{n-2}...g_2] = e,
\] 
for all $1\leq j \leq n-1$. 
We again use the above relation for $g_2$ and the relation obtained in (\ref{cyclic:1}), to get 
\[
g_2g_1g_0g_{n-1}g_{n-2}...g_3 = g_0g_{n-1}g_{n-2}...g_1.
\]
Continuing with this process, we inductively get 
\begin{eqnarray}\label{cyclic ration}
g_j...g_0g_{n-1}g_{n-2}...g_{j+1} = g_0g_{n-1}g_{n-2}...g_1.
\end{eqnarray}
for all $j = 1,2,...,n-1$. 
The equal products that apper in \eqref{cyclic ration} will appear numerous times in the following, so we denote it from now on as
\begin{equation}\label{definition of cyclic}
	\mathfrak{g} := g_{n-1}\ldots g_1g_0.
\end{equation}
This proves that the position of the contractible strip $S$ that we added, does not create any ambiguity in the 
computation of the fundamental group. 
It is clear that $\pi_1(U_{0,1}\cap U_{0,2}, \overline{\epsilon})$ is infinite cyclic generated by $ g_0g_{n-1}\ldots g_1$.
By the  Seifert-Van Kampen Theorem,  we get 
 \begin{eqnarray}\label{Fundamental group of U0}\nonumber
 \pi_1(U_{0}, \overline{\epsilon}) &=& \pi_1(U_{0,1}, \overline{\epsilon}) *_{\pi_1(U_{0,1} \cap U_{0,2})}\pi_1(U_{0,2}, \overline{\epsilon})  \\ 
 &=& \langle g_0,...,g_{n-1},g_0',...,g_{n-1}'| \ (\ref{cyclic ration}) \ \text{ holds for $1\leq i \leq n-1$}\rangle.
 \end{eqnarray}
\subsection{Fundamental group of $U_0-V(y)$}\label{fundamental group of U0-V(Y)}

 We use the projection defined in (\ref{original_proj}) to compute the fundamental group $\pi_1(U_{0}-V(y), \overline{\epsilon})$. Notice that the restriction of the map $Pr$ to $U_0-V(y)$ 
 \[
 Pr|_{U_{0}-V(y)} : U_{0}-V(y) \longrightarrow B_{2\varepsilon}(0)-\{0\}
 \] 
 as defined in (\ref{original_proj}), is a fibration with each fiber homeomorphic to $\mathbb{C}-\{2n\ \text{points}\}$. 
 Therefore using the Lemma \ref{crucial_lemma}, the fundamental group 
 \[
 \pi_1(U_{0}-V(y), \overline{\epsilon}) = \mathbb{Z} \ltimes F_{2n},
 \]
 where $\mathbb{Z}$ is generated by $\gamma_0 := \e \exp(2\pi it)$ (denoting the base curve) and $F_{2n}$ is a free group on $2n$ generators, namely, $g_j$ and $g_k'$ for $0\leq j, k \leq n-1$. 
 
 We now compute the monodromy action of $\gamma_0$ on $g_j$ and $g_k'$. In order to do so, we first denote by $\mathcal{F}_t$ to be the fiber of $Pr$ over $\e\exp(2\pi it)$. Note that the points
 \[
 \mathcal{F}_t \cap L_{y,j} = [1: \e \exp(2\pi it): \z_3^j], \; \; \text{and} \; \; \mathcal{F}_t \cap L_{x,j} = [1: \e \exp(2\pi it): \e \exp(2\pi it)\z_n^{-j}],
 \]
in $\frac{z}{x}$ coordinates, at $t =0, \frac{1}{4},\frac{1}{2},\frac{3}{4}$  are
\begin{eqnarray*}
\z_n^j, \ \z_n^j, \ \z_n^j, \ \z_n^j \quad \text{and} \quad \e \z_n^{-j}, \e i \z_n^{-j}, -\e \z_n^{-j}, \ -\e i \z_n^{-j}
\end{eqnarray*} 
respectively. This, in particular indicates that in the fibre $\mathcal{F}_t$, the $g_j$ appears clockwise while $g_k'$ appears counterclockwise. Action of $\gamma_0$ on the fibre is as follows:
\begin{figure}[H]
\begin{center}
\raggedleft
\begin{minipage}{13cm}
  \includegraphics[width=0.6\linewidth]{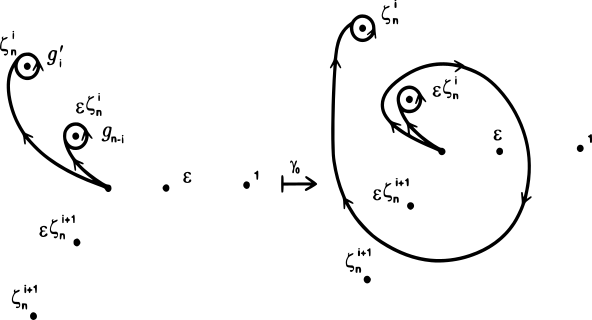}
  \end{minipage}
    \caption{Action of $\gamma_0$}
      \label{fig:action for gamma-0-general case.}
\end{center}
\end{figure}
From this figure we get,
\begin{eqnarray}\label{Action of gamma-0-general}
\gamma_0 . g_i &=& g_i \\ \nonumber
\gamma_0 . g_0'&=& g_0^{-1}g_1^{-1}...g_{n-1}^{-1}g_0^{-1}g_0'g_0g_{n-1}...g_1g_0, \\ \nonumber
\gamma_0 . g_i'&=& g_{n-i}^{-1}g_{n-i+1}^{-1}...g_{n-1}^{-1}g_0^{-1}g_1^{-1}...g_{n-i-1}^{-1}g_{n-i}^{-1}g_i'\times  \\
&& \times g_{n-i}g_{n-i-1}...g_1g_0g_{n-1}...g_{n-i+1}g_{n-i}, \nonumber
\end{eqnarray}
for $1 \leq i \leq n-1$.

By \eqref{cyclic ration}, we have proven the following lemma.

\begin{lemma}\label{lemma:fund_group_of_U0_minus_y}
	The fundamental group of $U_0-V(y)$ is the semi-direct product $\mathbb{Z}\ltimes F_{2n}$, where the free group $F_{2n}$ is generated by $\{g_0, \dots, g_{n-1}, g'_0,\dots, g'_{n-1}\}$, the free cyclic group is generated by $\gamma_0$, and we have 
	\begin{eqnarray}\label{Action of gamma-0-general-simplified}
		\gamma_0^{-1} g_i \gamma_0 &=& g_i \\ \nonumber
		\gamma_0^{-1} g_i' \gamma_0 &=& \mathfrak{g}^{-1}g_{n-i}^{-1}g_i'g_{n-i}\mathfrak{g}, \nonumber
	\end{eqnarray}
	as denoted in \eqref{definition of cyclic}.
\end{lemma}

\begin{rmk}\label{rmk:modular_indices}
	Writing $g_M$, $g_N'$, and $g_S''$ (for $M,N,S \geq n$) means $g_{\overline{M}}$, $g_{\overline{N}}'$ and $g_{\overline{S}}''$. Here $\overline{M} := M\; \text{mod} \; n$, $\overline{N} := N\; \text{mod} \; n$, and $\overline{S} := S\; \text{mod} \; n$.
\end{rmk}

\subsection{Fundamental group of $\mathbb{CP}^2-(\mathcal{C}\cup V(x))$}\label{fundamental group of CP2-(C U V(X))}
Now, we consider open connected subsets $U_{0}$ and $\mathbb{CP}^2-(\mathcal{C}\cup V(x) \cup V(y))$ of $\mathbb{CP}^2-\mathcal{C}$ and use the Seifert-Van Kampen Theorem to derive the fundamental group of $\mathbb{CP}^2-(\mathcal{C}\cup V(x))$. To accomplish this, we utilize the fundamental group computed in Sections \ref{fundamental group of U0} and  \ref{fundamental group of U0-V(Y)} and observe that 
\begin{eqnarray*}
    U_{0} \cup (\mathbb{CP}^2-(\mathcal{C}\cup V(x) \cup V(y))) &=& \mathbb{CP}^2-(\mathcal{C}\cup V(x)), \; \text{and} \\
    U_{0} \cap (\mathbb{CP}^2-(\mathcal{C}\cup V(x) \cup V(y))) &=& U_{0}-V(y).
\end{eqnarray*}
Before proceeding further, we define $g_j''$ for $0 \leq j \leq n-1$ as follows: 
\begin{eqnarray}\label{Defn of gi''}
g_{n-j}''(t) :=
         \begin{cases}
            [1: \e\exp(2\pi itj/n): 0] \quad 0 \leq t \leq 1 \\
            [1: \z_n^j(2\e+\delta -1)+ t(\z_n^j(1-\delta - \e)) : 0] \quad 1 \leq t \leq 2 \\
            [1: \z_n^j(1 - \delta\exp(2\pi i(t-2))): 0] \quad 2 \leq t \leq 3 \\
            [1: \z_n^j(2\e+\delta -1)+ (5-t)(\z_n^j(1-\delta - \e)) : 0] \quad 3 \leq t \leq 4 \\
            [1: \e\exp(2\pi i(5-t)j/n): 0] \quad 4 \leq t \leq 5.
         \end{cases}
\end{eqnarray}
Note that $g_{n-j}''$ for $1\leq j \leq n-1$ is a loop around $L_{z,{n-j}}$ and $g_0'' := g_n''$ is a loop around $L_{z,0}$. \\

\begin{Def}
    We define $\gamma_{j+1} := Pr(g_{n-j}'') \subset \mathbb{C}$ for $j=0, 1, ..., n-1$, where $Pr$ is the projection defined in (\ref{original_proj}).
\end{Def}
Now, we use the Seifert-Van Kampen Theorem, to get
\begin{eqnarray}\label{Fundamental group of U_0 part}
\pi_1(\mathbb{CP}^2-(\mathcal{C}\cup V(x)),\overline{\epsilon}) = \
\langle g_i,g_i',g_i''\ (0\leq i \leq n-1)\ | \ (\ref{cyclic ration})\ \text{holds for}
 \ 1\leq i \leq n-1, \\ \nonumber \text{Relations $R_1$ from Definition \ref{Definition of $R_1$}} \rangle,
\end{eqnarray}
We want to compute the relations in $R_1$ induced by the monodromy action of (\ref{Fundamental group of $U$}), but before that we will need to describe an embedding of $\gamma_0$ inside $\mathbb{CP}^2 - (\mathcal{C}\cup V(x))$ in terms of $g_0,\dots,g_{n-1}$.

\begin{lemma}\label{lemma:gamma_0_image}
	Let 
	$$\text{in}:B_{2\epsilon}(0)-\{0\} \to \mathbb{CP}^2 - (\mathcal{C}\cup V(x) \cup V(y))$$
	be the inclusion $ t\mapsto [1:t:0]$, then 
	\begin{equation}\label{eq:gamma_0_is_g}
		\text{in}_*(\gamma_0)=\mathfrak{g},
	\end{equation}
	as defined in \eqref{definition of cyclic}.
\end{lemma}
\begin{proof}
	We pick a continuous deformation $H_t$ of lines through $\overline{\epsilon}$ from the line $z=0$ to the fiber of $ Pr^{-1}(\e) $, such that each $H_t$ intersects all the lines $L_{x,i}$.
	Such a deformation exists since $\mathbb{CP}^1$ is the parametrization of the set of all complex lines through a basepoint. Only finitely many lines among them do not intersect $L_{x,i}$, therefore any path in the complement of these finite number of points in $\mathbb{CP}^1$ corresponds to a deformation that satisfies the required property.  

Note that $L_{x,i} \cap H_t$ defines a path in $\mathbb{CP}^2$, which is compact and bounded (being a continuous 
image of interval $[0,1]$). We thus can continuously deform first $\text{in}_*\gamma_0$ in $\mathbb{CP}^2$ in such a way that it lies in $``z=0"$, homotopic to $[1:\e \exp(2\pi it):0]$ and bounds all the $n$ paths (whose boundary is being denoted as points in the right figure) as shown below (Figure \ref{Inclusion of gamma-0.}).

\begin{figure}[H]
	\begin{center}
		\raggedleft
		\begin{minipage}{15cm}
			\includegraphics[width=0.8\linewidth]{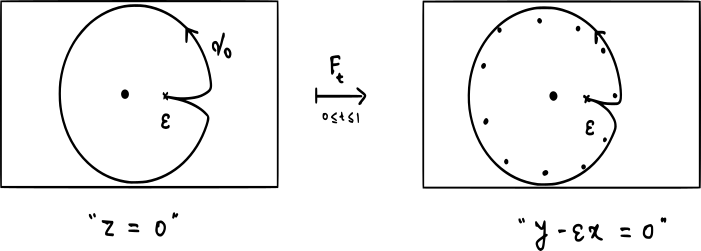}
		\end{minipage}
		\caption{Inclusion of $\gamma_0$ in $\mathbb{CP}^2$}
		\label{Inclusion of gamma-0.}
	\end{center}
\end{figure}

It is then clear that $\text{in}_*\gamma_0 = g_0g_{n-1}\cdots g_1 = \mathfrak{g}$. 
\end{proof}
\begin{rmk}
	Note that the different choices of deformation only gives 
	\[ \text{in}_*\gamma_0 = g_jg_{i-1}...g_0g_{n-1}g_{n-2}...g_{j+1},
	\]
for different $j$. But in view of (\ref{cyclic ration}), there is no ambiguity. 
\end{rmk}

\begin{prop}\label{Prop for R1 relations}
The fundamental group $\pi_1(\mathbb{CP}^2-(\mathcal{C}\cup V(x) \cup V(y)), \overline{\epsilon})$ is generated by the $3n+1$ generators $\{ \gamma_0, g_0, \dots, g_{n-1}, g'_0, \dots, g'_{n-1}, g''_0, \dots, g''_{n-1}\}$ subject to the following relations:
\begin{eqnarray*}
\gamma_0^{-1}g_i \gamma_0 &=&  g_i  \\ 
\gamma_0^{-1}g_i' \gamma_0 &=& \mathfrak{g}^{-1}g_{n-i}^{-1}g_i'g_{n-i}\mathfrak{g}\\ 
 g_{n-k+1}''^{-1} g_{n-i} g_{n-k+1}'' &=& g_{n-i}g_{n-i-1}...g_{n-i-(k-1)}g_{i+(k-1)}'g_{n-i-(k-1)}^{-1}...g_{n-i-1}^{-1}g_{n-i}\cdot \\ 
&& g_{n-i-1}...g_{n-i-(k-1)}g_{i+(k-1)}'^{-1}g_{n-i-(k-1)}^{-1}...g_{n-i-1}^{-1}g_{n-i}^{-1}\\ 
 g_{n-k+1}''^{-1} g_{i}' g_{n-k+1}'' &=& \mathscr{G}^k_{n-i}\mathscr{G}^k_{n-i+1}\cdots \mathscr{G}^k_{n-i+(k-2)} \mathscr{H}^k_{n-i+(k-1)} g_i'\cdot \\
&& {\mathscr{H}^k_{n-i+(k-1)}}^{-1} {\mathscr{G}^k_{n-i+(k-2)}}^{-1} \cdots {\mathscr{G}^k_{n-i+1}}^{-1} {\mathscr{G}^k_{n-i}}^{-1},
 \end{eqnarray*}
 for $1 \leq k \leq n$ and $0 \leq i \leq n-1$ (the definitions of $\mathscr{G}^k_{n-s}$ and $\mathscr{H}^k_{n-i+(k-1)}$ are given 
 in (\ref{def of k-conjugate}) and (\ref{def of half k-conjugate}) respectively, and the indices follow from Remark \ref{rmk:modular_indices}). 
\end{prop}
\begin{proof}[Proof of Lemma \ref{Prop for R1 relations}]
It is clear from the definition of $\gamma_i$ that the fundamental group of the base space of the projection map 
$Pr$ (from (\ref{original_proj})) is generated by $\gamma_0,\gamma_1,...,\gamma_{n-1},\gamma_n$. 
Therefore, we compute the monodromy action of these on the fibers of $Pr$.  
The action of $\gamma_0$ was computed in Section \ref{fundamental group of U0-V(Y)}, see Lemma \ref{lemma:fund_group_of_U0_minus_y}.
Action of $\gamma_1$ on the fiber is as shown in the below figure
\begin{center}
\begin{figure}[H]
\raggedleft
\begin{minipage}{15cm}
  \includegraphics[width=0.7\linewidth]{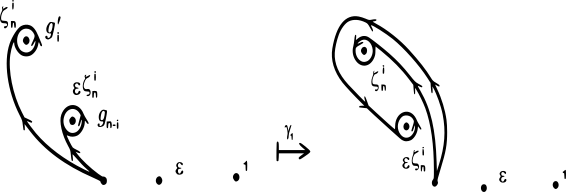}
  \caption{Action of $\gamma_1$}
  \label{fig:action for gamma-1-general case.}
  \end{minipage}
\end{figure}
\end{center}
Therefore, we get for each $0 \leq i \leq n-1$,
\begin{eqnarray}
\gamma_1 \cdot g_{n-i} = g_0''^{-1}g_{n-i}g_0'' &=& g_{n-i}g_i'g_{n-i}g_i'^{-1}g_{n-i}^{-1} \\ \nonumber
\gamma_1 \cdot g_i'= g_0''^{-1}g_{i}'g_0'' &=& g_{n-i}g_i'g_{n-i}^{-1}. \nonumber
\end{eqnarray} 
Similarly, we compute the action of $\gamma_2$ on the fibre and from the below figure 

\begin{figure}[H]
\begin{center}
\raggedleft
\begin{minipage}{15cm}
  \includegraphics[width=0.7\linewidth]{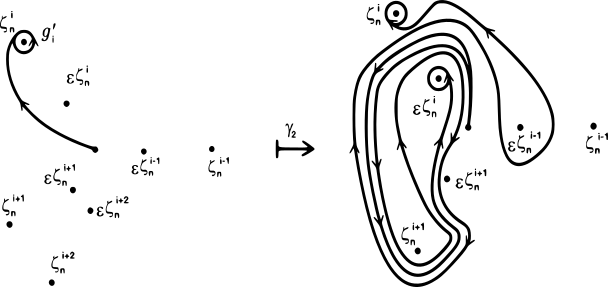}
  \end{minipage}
    \caption{Action of $\gamma_2$}
      \label{Base space.}
\end{center}
\end{figure}
we note that,
\begin{eqnarray}
\gamma_2 \cdot g_{n-i}= g_{n-1}''^{-1}g_{n-i}g_{n-1}'' &=& g_{n-i}g_{n-i-1}g_{i+1}'g_{n-i-1}^{-1}g_{n-i}g_{n-i-1}g_{i+1}'^{-1}g_{n-i-1}^{-1}g_{n-i}^{-1}\\ \nonumber
\gamma_2 \cdot g_i'= g_{n-1}''^{-1}g_{i}'g_{n-1}'' &=& g_{n-i}g_{n-i-1}g_{i+1}'g_{n-i-1}^{-1}g_{n-i}^{-1}g_{n-i-1}g_{i+1}'^{-1}g_{n-i-1}^{-1}g_{n-i}^{-1}g_{n-i+1} 
g_{n-i}g_i'\cdot \\ \nonumber
&& g_{n-i}^{-1}g_{n-i+1}^{-1}g_{n-i}g_{n-i-1}g_{i+1}'g_{n-i-1}^{-1}g_{n-i}g_{n-i-1}g_{i+1}'^{-1}g_{n-i-1}^{-1}g_{n-i}^{-1},\nonumber
\end{eqnarray} 
for each $0 \leq i \leq n-1$.
Action of $\gamma_3$ can be computed in similar fashion by looking at the monodromy action of $\gamma_3$ on the fibers of $Pr$, as in the below figures.
\begin{figure}[H]
\begin{center}
\raggedleft
\begin{minipage}{15cm}
  \includegraphics[width=0.8\linewidth]{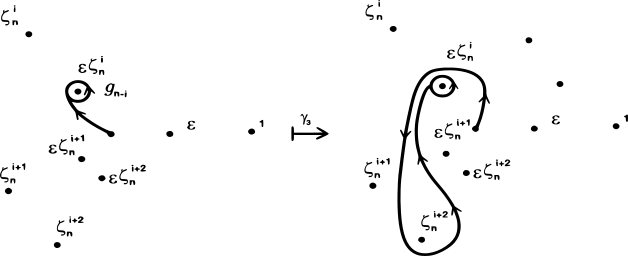}
  \end{minipage}
    \caption{Action of $\gamma_3$ on $g_i$}
      \label{action of gamma-3.}
\end{center}
\end{figure}

\begin{figure}[H]
\begin{center}
\raggedleft
\begin{minipage}{15cm}
  \includegraphics[width=0.8\linewidth]{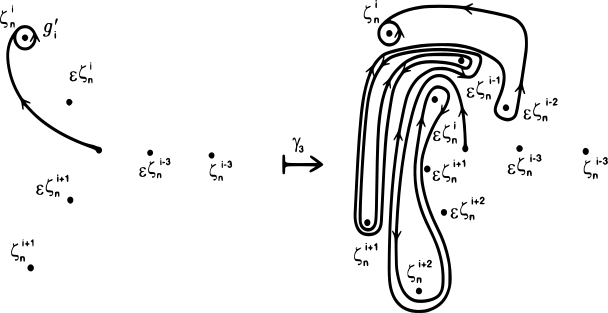}
  \end{minipage}
    \caption{Action of $\gamma_3$ on $g_i'$}
      \label{action of gamma-3.}
\end{center}
\end{figure}
We therefore get 
\begin{eqnarray}
\gamma_3 . g_i' = g_{n-2}''^{-1}g_{i}'g_{n-2}'' &=& g_{n-i}g_{n-i-1}g_{n-i-2}g_{i+2}'g_{n-i-2}^{-1}g_{n-i-1}^{-1}g_{n-i}^{-1}g_{n-i-1}g_{n-i-2}g_{i+2}'^{-1}g_{n-i-2}^{-1}\cdot \\ \nonumber
&& g_{n-i-1}^{-1}g_{n-i}^{-1} g_{n-i+1} g_{n-i}g_{n-i-1}g_{i+1}'g_{n-i-1}^{-1}g_{n-i}^{-1}g_{n-i+1}^{-1}g_{n-i}g_{n-i-1}g_{i+1}'^{-1}\cdot \\ \nonumber
&& g_{n-i-1}^{-1}g_{n-i}^{-1}g_{n-i+1}^{-1}g_{n-i+2}g_{n-i+1}g_{n-i}g_i'g_{n-i}^{-1}g_{n-i+1}^{-1}g_{n-i+2}^{-1}g_{n-i+1}g_{n-i}\cdot \\ \nonumber
&& g_{n-i-1}g_{i+1}'g_{n-i-1}^{-1}g_{n-i}^{-1}g_{n-i+1}g_{n-i}g_{n-i-1}g_{i+1}'^{-1}g_{n-i-1}^{-1}g_{n-i}^{-1}g_{n-i+1}^{-1}g_{n-i}\cdot \\ \nonumber
&& g_{n-i-1}g_{n-i-2}g_{i+2}'g_{n-i-2}^{-1}g_{n-i-1}^{-1}g_{n-i}g_{n-i-1}g_{n-i-2}g_{i+2}'^{-1}g_{n-i-2}^{-1}g_{n-i-1}^{-1}g_{n-i}^{-1}\\ \nonumber
\gamma_3 . g_{n-i} = g_{n-2}''^{-1}g_{n-i}g_{n-2}'' &=& g_{n-i}g_{n-i-1}g_{n-i-2}g_{i+2}'g_{n-i-2}^{-1}g_{n-i-1}^{-1}g_{n-i}g_{n-i-1}g_{n-i-2}g_{i+2}'^{-1}g_{n-i-2}^{-1}\\ \nonumber
&& g_{n-i-1}^{-1}g_{n-i}^{-1}. \nonumber   
\end{eqnarray}
On the basis of the above discussion, we can write down the equation for the action of $\gamma_k$ more generally as follows:
\begin{eqnarray}\label{action of general gamma-k}
\gamma_k . g_{n-i} &=& g_{n-k+1}''^{-1} g_{n-i} g_{n-k+1}'' \\ \nonumber
&=& g_{n-i}g_{n-i-1}...g_{n-i-(k-1)}g_{i+(k-1)}'g_{n-i-(k-1)}^{-1}...g_{n-i-1}^{-1}g_{n-i}\cdot \\ \nonumber
&& g_{n-i-1}...g_{n-i-(k-1)}g_{i+(k-1)}'^{-1}g_{n-i-(k-1)}^{-1}...g_{n-i-1}^{-1}g_{n-i}^{-1}\\ \nonumber
\gamma_k . g_i' &=& g_{n-k+1}''^{-1} g_{i}' g_{n-k+1}'' \\ \nonumber
&=& \mathscr{G}^k_{n-i}\mathscr{G}^k_{n-i+1}\cdots \mathscr{G}^k_{n-i+(k-2)} \mathscr{H}^k_{n-i+(k-1)} g_i' 
{\mathscr{H}^k_{n-i+(k-1)}}^{-1} {\mathscr{G}^k_{n-i+(k-2)}}^{-1} \cdots {\mathscr{G}^k_{n-i+1}}^{-1} {\mathscr{G}^k_{n-i}}^{-1}
\end{eqnarray}
for $1 \leq k \leq n$ and $0 \leq i \leq n-1$. Here we define \textit{k-conjugate} $\mathscr{G}^k_{n-s}$ for $s=i,i-1,...,i-(k-2)$ 
and $\mathscr{H}^k_{n-i+(k-1)}$ as follows:
\begin{eqnarray}\label{def of k-conjugate}
    \mathscr{G}^k_{n-s} &:=& g_{n-s}g_{n-s-1}\cdots g_{n-s-(k-1)}g_{s+(k-1)}'g_{n-s-(k-1)}^{-1}\cdots g_{n-s}^{-1} \cdot \\ \nonumber
   && g_{n-s-1}\cdots g_{n-s-(k-1)}g_{s+(k-1)}'^{-1}g_{n-s-(k-1)}^{-1}\cdots g_{n-s}^{-1},
\end{eqnarray}
and 
\begin{eqnarray}\label{def of half k-conjugate}
    \mathscr{H}^k_{n-i+(k-1)} := g_{n-i+(k-1)}g_{n-i+(k-2)}\cdots g_{n-i}. 
\end{eqnarray}
\\ 
\begin{comment}
\begin{eqnarray}
\mathcal{G}_i &:=& g_{n-i}g_{n-i-1}...g_{n-i-(k-1)}g_{i+(k-1)}'g_{n-i-(k-1)}^{-1}...g_{n-i}^{-1}g_{n-i-1}...g_{n-i-(k-1)} \\ \nonumber
&& g_{i+(k-1)}'^{-1}g_{n-i-(k-1)}^{-1}...g_{n-i-1}^{-1}g_{n-i}^{-1} \\ \nonumber
&& ... \\ \nonumber 
&& g_{n-i+(s-1)}...g_{n-i}g_{n-i-1}...g_{n-i-(k-s)}g_{i+(k-s)}' g_{n-i-(k-s)}^{-1}...g_{n-i}^{-1}...g_{n-i+(s-1)}^{-1}g_{n-i+(s-2)} \\ \nonumber
&& ...g_{n-i}g_{n-i-1}...g_{n-i-(k-s)}g_{i+(k-s)}'g_{n-i-(k-s)}^{-1}...g_{n-i}^{-1}g_{n-i+1}^{-1}...g_{n-i+(s-1)}^{-1} \\ \nonumber
&& ... \\ \nonumber
&&  g_{n-i+(k-2)}g_{n-i+(k-3)}...g_{n-i}g_{n-i-1}...g_{n-i-(k-2)}g_{i+(k-2)}' g_{n-i-(k-2)}^{-1}...g_{n-i}^{-1}...g_{n-i+(k-2)}^{-1} \\\nonumber
&& g_{n-i+(k-3)}...g_{n-i}g_{n-i-1}...g_{n-i-(k-2)}g_{i+(k-2)}'g_{n-i-(k-2)}^{-1}...g_{n-i}{-1}g_{n-i+1}^{-1}...g_{n-i+(k-2)}^{-1} \\ \nonumber 
&& g_{n-i+(k-1)}...g_{n-i}. \nonumber
\end{eqnarray}
\end{comment}
After substituting the definition of $\mathfrak{g}$, and using \eqref{cyclic ration}, this completes the proof of the Lemma.
\end{proof} 

\begin{cor}\label{cor:fund_group_CP2_minus_C_minus_x}
	The fundamental group $\pi_1(\mathbb{CP}^2 - (\mathcal{C}\cup V(x)), \bar{\epsilon})$ is generated by the $3n$ generators $\{g_0,\dots, g_{n-1}, g'_0, \dots, g'_{n-1}, g''_0, \dots, g''_{n-1}\}$ subject to the relations
	\begin{eqnarray*}
		\left[\mathfrak{g}, g_i\right] &=& 1  \\ 
		\left[ g'_i, g_{n-i} \right] &=& 1 \\
		g_{n-k+1}''^{-1} g_{n-i} g_{n-k+1}'' &=& g_{n-i}g_{n-i-1}...g_{n-i-(k-1)}g_{i+(k-1)}'g_{n-i-(k-1)}^{-1}...g_{n-i-1}^{-1}g_{n-i}\cdot \\ 
		&& g_{n-i-1}...g_{n-i-(k-1)}g_{i+(k-1)}'^{-1}g_{n-i-(k-1)}^{-1}...g_{n-i-1}^{-1}g_{n-i}^{-1}\\ 
		g_{n-k+1}''^{-1} g_{i}' g_{n-k+1}'' &=& \mathscr{G}^k_{n-i}\mathscr{G}^k_{n-i+1}\cdots \mathscr{G}^k_{n-i+(k-2)} \mathscr{H}^k_{n-i+(k-1)} g_i'\cdot \\
		&& {\mathscr{H}^k_{n-i+(k-1)}}^{-1} {\mathscr{G}^k_{n-i+(k-2)}}^{-1} \cdots {\mathscr{G}^k_{n-i+1}}^{-1} {\mathscr{G}^k_{n-i}}^{-1}.
	\end{eqnarray*}
\end{cor}

\begin{proof}
	We use Seifert-Van Kampen Theorem on $\pi_1(U_0, \bar{\epsilon})$ and $\pi_1(\mathbb{CP}^2 - (\mathcal{C}\cup V(x) \cup V(y)), \bar{\epsilon})$.
	By Lemma \ref{lemma:gamma_0_image}, this amounts to setting $\gamma_0$ equal to $\mathfrak{g}$ in the relations defining $\pi_1(\mathbb{CP}^2 - (\mathcal{C}\cup V(x) \cup V(y)), \bar{\epsilon})$ as described in Lemma \ref{Prop for R1 relations}, so we get the desired result.
\end{proof}

\begin{rmk}
	It is possible to use Seifert-Van Kampen Theorem in order to compute the fundamental group $\pi_1(\mathbb{CP}^2 - (\mathcal{C}\cup V(x)), \bar{\epsilon})$ without the precise knowledge of the relations in $\pi_1(U_0 - V(y), \bar{\epsilon})$, only using the fact it is generated by the same generators as $\pi_1(U_0, \bar{\epsilon})$ and $\pi_1(\mathbb{CP}^2 - (\mathcal{C}\cup V(x) \cup V(y)), \bar{\epsilon})$.
	None the less, the computations in Section \ref{fundamental group of U0-V(Y)} were used in the proof of Lemma \ref{Prop for R1 relations}, so no extra work was performed here.
\end{rmk}

\subsection{Fundamental group of $U_\infty$}\label{Fundamental group of Uinfinity}
To compute the fundamental group $\pi_1(U_{\infty}, \bar{\epsilon})$ we will use another symmetry in the definition of $\mathcal{C}$, namely the fact it is symmetric to permutation of the variables.
Recall that, 
\[
U_\infty = \{[x:y:z] \in \CP-\C \ | \ y\neq 0, \ \left|\frac{x}{y}\right| < 2\e \}.
\]
We choose the basepoint in this case to be $\epsilon := [\e: 1:0]$.
We can then immediately see that the map 
\begin{align*}
	\phi: U_0 & \to U_{\infty} \\
	[x:y:z] & \mapsto [y:x:z]
\end{align*}
is an isomorphism that sends the basepoint $\bar{\epsilon}$ to the basepoint $\epsilon$.
Denoting 
\begin{eqnarray*}
	g_{\infty, n-i}' := \phi(g_i) \\
	g_{\infty, n-j} := \phi(g_j'),
\end{eqnarray*}
we can thus see that $\pi_1(U_{\infty}, \epsilon)$ is generated by the $2n$ generators $\{g_{\infty, 0}, \dots, g_{\infty, n-1}, g'_{\infty, 0}, \dots, g'_{\infty, n-1}\}$ subject to the relation (which is the push-forward of \eqref{cyclic ration})
\begin{equation}\label{eq:rel_U_infty}
	g'_{\infty, j}...g'_{\infty, n-1}g'_{\infty, 0}g'_{\infty, 1}...g'_{\infty, j-1} = g'_{\infty, 0}g'_{\infty, 1}...g'_{\infty, n-1}.
\end{equation}
Similarly, $\pi_1(U_{\infty}-V(x), \epsilon)$ is generated by the $2n+1$ generators $\{\gamma_{\infty}, g_{\infty, 0}, \dots, g_{\infty, n-1}, g'_{\infty, 0}, \dots, g'_{\infty, n-1}\}$ subject to 

\begin{eqnarray}\label{eq:rel_U_infty_minus_x}
	\gamma_{\infty}^{-1} g'_{\infty, i} \gamma_{\infty} &=& g'_{\infty, i} \\ \nonumber
	\gamma_{\infty}^{-1} g_{\infty, i} \gamma_{\infty} &=& \mathfrak{g}_{\infty}^{-1}g_{\infty, n-i}'^{-1}g_{\infty, i}g_{\infty, n-i}'\mathfrak{g}_{\infty}, \nonumber
\end{eqnarray}
where we have denoted $\mathfrak{g}_{\infty} = g'_{\infty, 0}\dots g'_{\infty, n-1}$ and $\gamma_{\infty}:=\phi(\gamma_0)$.

\section{Fundamental group of $\mathbb{CP}^2-\mathcal{C}$}\label{final computation}

Now to finally compute the fundamental group $\pi_1(\mathbb{CP}^2-\mathcal{C})$, we use the Seifert-Van Kampen Theorem with open covers as $\mathbb{CP}^2-(\mathcal{C}\cup V(x))$ and $U_\infty$. To do this, 
we first need to translate the basepoint in $\pi_1(U_\infty,\epsilon)$ and in $\pi_1(U_\infty-V(x),\epsilon)$ to obtain $\pi_1(U_\infty,\overline{\epsilon})$ and $\pi_1(U_\infty-V(x),\overline{\epsilon})$ respectively. This we do by defining a path $\alpha$ between the two basepoints $\overline{\epsilon} = [1: \e : 0]$ and $\epsilon = [1: 1/\e : 0]$ as follows.

\[
\alpha(t) := \begin{cases}
                  [1: \e + (1-\delta-\e)t : 0], \quad 0 \leq t \leq 1\\
                  [1: 1 + \delta\exp(\pi it) : 0],  \quad 1 \leq t \leq 2 \\
                  [1: 1 + \delta + (\e^{-1} - 1 - \delta)(t-2) : 0], \quad 2 \leq t \leq 3.
                  \end{cases}
\]

Note that the loops $\alpha*g_{\infty,i}*\alpha^{-1}$, $\alpha*g_{\infty,i}'*\alpha^{-1}$ and 
$\alpha*g_{\infty,i}''*\alpha^{-1}$ are all based at the basepoint $\overline{\epsilon}$. Here $``*"$ represents concatenation.

\begin{comment}
    \textbf{Fundamental group of ($\mathbb{CP}^2-(\mathcal{C}\cup V(y))$):}
Now we apply the Van Kampen Theorem on the open sets $U_\infty$ and 
$(\mathbb{CP}^2-(\mathcal{C}\cup V(x) \cup V(y))$. We have the following commutative diagram:

\begin{center}
\begin{tikzcd}
\pi_1(U_{\infty}-V(x), \epsilon) \arrow{r}{\text{$I_x$}} \arrow[swap]{d}{\text{$J_x$}} & \pi_1(U_{\infty}, \epsilon) \arrow{d}{I_\infty} \\%
\pi_1(\mathbb{CP}^2-(\mathcal{C}\cup V(x) \cup V(y)), \epsilon) \arrow{r}{Pr}& \pi_1(\mathbb{CP}^2-\mathcal{C}\cup V(y), \epsilon).
\end{tikzcd}
\end{center}
From this we get 
\[
\pi_1(\mathbb{CP}^2-\mathcal{C}\cup V(y), \epsilon) = 
\pi_1(\mathbb{CP}^2-(\mathcal{C}\cup V(x) \cup V(y)), \epsilon) *_{\pi_1(U_{\infty}-V(x), \epsilon)}  \pi_1(U_{\infty}, \epsilon),
\]
and therefore
\begin{eqnarray}\label{Fundamental group of U_infity part}
\pi_1(\mathbb{CP}^2-(\mathcal{C}\cup V(y)),\epsilon) = \
\langle g_{\infty,i},g_{\infty,i}',g_{\infty,i}'',\gamma_\infty\ (0\leq i \leq n-1) \ | \ \text{Satisfying the relations} \ (\ref{Action-of-g-infinity}), \\ \nonumber
 %\\ \nonumber
 \ \gamma_\infty \cdot g_{\infty,i},\  g_{\infty,i}'' \cdot g_{\infty,j} \ \text{and} \ g_{\infty,i}'' \cdot g_{\infty,j}' \ (\text{for}\ 0\leq i,j \leq n-1), \\ \nonumber
 \ \text{$I_x(\gamma_\infty)J_x(\gamma_\infty)^{-1} \ (= J_x(\gamma_\infty)^{-1})$} \rangle.
\end{eqnarray}
\end{comment}

\begin{figure}[H]
\begin{center}
\raggedleft
\begin{minipage}{15cm}
  \includegraphics[width=0.7\linewidth]{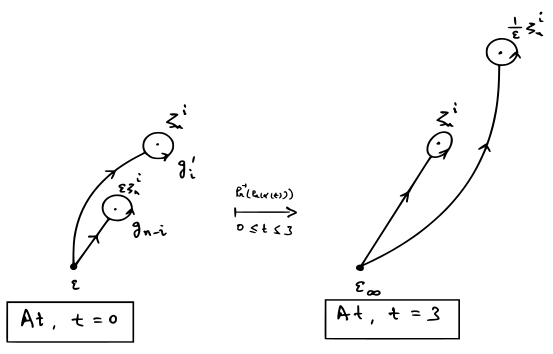}
  \end{minipage}
    \caption{Change of base}
      \label{Change-of-base.}
\end{center}
\end{figure}
Figure \ref{Change-of-base.} shows the monodromy action on the fibre of $Pr$ while traversing the path defined by $\alpha(t)$. 
Also from the defining equation of $g_{\infty,i} = \phi(g_{n-i}')$ it is clear that $g_{\infty, i}$ is a loop around the point 
$(\frac{1}{\varepsilon})\z_n^i$ (in $z/x$ coordinate) and based at $\frac{1}{\varepsilon}$. Similarly from the defining equation of $g_{\infty, i}'$ it is clear that $g_{\infty, i}'$ is a loop around $\z_n^i$ (in $z/x$ coordinate). Therefore from Figure \ref{Change-of-base.} it follows that 
\[
\alpha*g_{\infty, i}*\alpha^{-1} \quad \text{and} \quad \alpha*g_{\infty, n-i}'*\alpha^{-1}
\]
are as shown (on the right side) in Figure \ref{After base Changing}.

\begin{figure}[H]
\begin{center}
\raggedleft
\begin{minipage}{15cm}
  \includegraphics[width=0.7\linewidth]{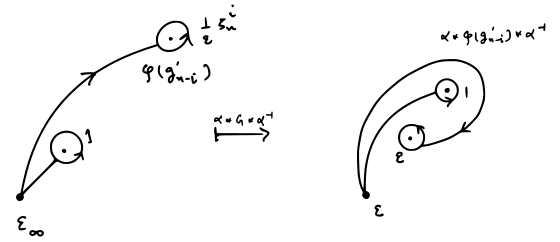}
  \end{minipage}
    \caption{After base changing}
      \label{After base Changing}
\end{center}
\end{figure} 
Now, it is not hard to see that 
\begin{eqnarray*}
\alpha*g_{\infty, i}*\alpha^{-1} &=& g_{n-i}'^{-1}g_{i}g_{n-i}', \;  \text{and}\\
\alpha*g_{\infty, i}'*\alpha^{-1} &=& g_i'.
\end{eqnarray*}
We translate the relation we have obtained for $\pi_1(U_\infty, \epsilon)$ and $\pi_1(U_{\infty} - V(x), \epsilon)$, by $\alpha$, to get
\begin{equation}\label{eq:rel_U_infty_translated}
g'_{j}...g'_{n-1}g'_{0}g'_{1}...g'_{j-1} = g'_{0}g'_{1}...g'_{n-1}
\end{equation}
for all $1\le j\le n-1$.

Now applying the Seifert-Van Kampen Theorem with open covers $\mathbb{CP}^2-(\mathcal{C}\cup V(x))$ and $U_\infty$, we 
obtain the following:
\begin{prop}\label{prop:main}
$\pi_1(\mathbb{CP}^2-\mathcal{C}, \overline{\epsilon}) = \langle g_i ,g_i',g_i'' \ | \  \text{such that the following relations hold} \rangle:
$
\begin{eqnarray*}
%[g_0, \gamma_0] &=& e \\ \nonumber
[g_{n-i}, g_i'] &=& e \\ \nonumber
[g_i, g_{n-1}g_{n-2} \ldots  g_0] & = & e\\ \nonumber
[g_i', g_0'g_{1}' \ldots  g_{n-1}'] &=& e \\ \nonumber
\gamma_k . g_{n-i} = g_{n-k+1}''^{-1} g_{n-i} g_{n-k+1}'' &=& g_{n-i}g_{n-i-1}...g_{n-i-(k-1)}g_{i+(k-1)}'g_{n-i-(k-1)}^{-1}...g_{n-i-1}^{-1}g_{n-i}\cdot \\ 
&& g_{n-i-1}...g_{n-i-(k-1)}g_{i+(k-1)}'^{-1}g_{n-i-(k-1)}^{-1}...g_{n-i-1}^{-1}g_{n-i}^{-1}\\ 
%\gamma_k \cdot g_i'= g_{n-k+1}'' g_i' g_{n-k+1}''^{-1} &=& \mathcal{G}_ig_i'\mathcal{G}_i^{-1} \nonumber
\gamma_k . g_i' = g_{n-k+1}''^{-1} g_{i}' g_{n-k+1}'' &=& \mathscr{G}^k_{n-i}\mathscr{G}^k_{n-i+1}\cdots \mathscr{G}^k_{n-i+(k-2)} \mathscr{H}^k_{n-i+(k-1)} g_i'\cdot \\ 
&& {\mathscr{H}^k_{n-i+(k-1)}}^{-1} {\mathscr{G}^k_{n-i+(k-2)}}^{-1} \cdots {\mathscr{G}^k_{n-i+1}}^{-1} {\mathscr{G}^k_{n-i}}^{-1}.
\end{eqnarray*}
for $1 \leq k \leq n$ and $0 \leq i \leq n-1$, where the definition of $\mathscr{G}^k_{n-s}$ and $\mathscr{H}^k_{n-i+(k-1)}$ are given 
in (\ref{def of k-conjugate}) and (\ref{def of half k-conjugate}) respectively. 
\end{prop}
\begin{proof}
We saw in the beginning of Section \ref{Fundamental group of Uinfinity} that $\pi_1(U_{\infty} - V(X))$ is generated by 
$$g_{\infty, 0}, \dots, g_{\infty, n-1}, g'_{\infty, 0}, \dots, g'_{\infty, n-1}, \gamma_{\infty}$$
and it is evident that the images of each of those generators in $\pi_1(U_{\infty}, \bar{\epsilon})$ and in $\pi_1(\mathbb{CP}^2 - (\mathcal{C}\cup V(x)), \bar{\epsilon})$ are equal.
Thus by Seifert-Van Kampen Theorem, $\pi_1(\mathbb{CP}^2 - \mathcal{C}, \bar{\epsilon})$ is generated by $g_0,\dots, g_{n-1}, g'_0, \dots, g'_{n-1}, g''_0, \dots, g''_{n-1}$, subject to the relations outlined in Corollary \ref{cor:fund_group_CP2_minus_C_minus_x} and to the relations \eqref{eq:rel_U_infty_translated}, which is the statement of the Proposition.
	
\end{proof}
\begin{rmk}
    We now focus on the last two type of equations in the statement of the proposition, for small values of $k$.
    As $k$ increases the result becomes more complicated, we present here the result for $0\le k\le 3$.
    For example, for $\gamma_k . g_{n-i}$
\begin{eqnarray}
\text{when} \quad k=0  &\text{gives,}&  [g_{n-i}, g_1''g_{i-1}'] = e, \\
\text{when} \quad k=1  &\text{gives,}&  [g_{i}, g_0''] = e, \\
\text{when} \quad k=2  &\text{gives,}&  [g_{n-i}, g_{n-1}''g_{n-i}g_{i+1}'] = e\\
\text{when} \quad k=3  &\text{gives,}&  [g_{n-i}, g_{n-2}''g_{n-i}g_{n-i-1}g_{i+2}'g_{n-i-1}^{-1}] = e.
\end{eqnarray}
When $k=0$, (\ref{def of k-conjugate}) and (\ref{def of half k-conjugate}) translates to 
\[
\gamma_0g_i^{-1}\gamma_0^{-1} = \mathscr{G}^0_{i} \quad \text{and} \quad \mathscr{H}^0_{i} = \gamma_0
\]
respectively. Substituting these in the last equation of the Proposition \ref{prop:main}, gives 
\begin{eqnarray} \nonumber
    g_1''^{-1}g_i'g_1'' = g_{n-i-1}g_i'g_{n-i-1}^{-1} \\
\Rightarrow g_i'g_1''g_{n-i-1} = g_{n-i-1}g_i'g_1'' = g_1''g_{n-i-1}g_i'.
\end{eqnarray}
\end{rmk}

\begin{Def}\label{defn of G}
    Let $G$ denote the group generated by $g_i$ and $g_j'$ for $0 \leq i,j \leq n-1$ subject to the relations obtained above. That is 
    \begin{eqnarray*}
    G := \{g_i, g_j' \; | 0\le i,j \le n-1, [g_0, g_0'] = e, [g_i,\ g_0g_{n-1}g_{n-2}...g_1] = e, [g_{n-i}, g_i'] = e, 
    [g_i', g_0'g_{n-1}'\cdots g_1'] = e  \}.
    \end{eqnarray*}
\end{Def}
We can now prove the main theorem of the paper (Theorem \ref{final theorem}).
%\mytodo{State and prove that the fundamental group is semi-direct product of the free group (generated by $g_i''$) and the group defined above.}
\begin{proof}[Proof of Theorem \ref{final theorem}]
    Let $F_n := \mathbb{Z}\langle g_j'' \rangle$ denotes the free group with generators $g_j''$ for $0 \leq j \leq n-1$. It is clear 
    that $F_n$ and $G$ are subgroups of $\pi_1(\mathbb{CP}^2 - \mathcal{C})$ and in-fact $G$ is a normal subgroup in view of \ref{Prop for R1 relations} with the following short-exact sequence 
    \[
    1 \longrightarrow G \longrightarrow \pi_1(\mathbb{CP}^2 - \mathcal{C}) \longrightarrow F_n \longrightarrow 1.
    \]
    Therefore it follows that 
    \[
    \pi_1(\mathbb{CP}^2 - \mathcal{C}) = G \rtimes F_n.
    \]
\end{proof}

\bibliographystyle{plain}

\end{document}